\documentclass[a4paper,12pt]{amsart}
\usepackage{lineno,hyperref}

\usepackage{amssymb,amsmath,amsthm}

\usepackage[english]{babel}
\usepackage[T1]{fontenc}
\usepackage{enumitem}

\DeclareFontFamily{U}{wncy}{}
\DeclareFontShape{U}{wncy}{m}{n}{<->wncyr10}{}
\DeclareSymbolFont{mcy}{U}{wncy}{m}{n}
\DeclareMathSymbol{\letterche}{\mathord}{mcy}{"71} 

\newtheorem{cor}{Corollary}
\newtheorem{thm}[cor]{Theorem}
\newtheorem{lem}[cor]{Lemma}
\newtheorem{prop}[cor]{Proposition}

\theoremstyle{remark}
\newtheorem{rem}[cor]{Remark}

\newcommand{\C}{\Omega}
\newcommand{\Cn}{\C_{n}(a)}
\newcommand{\CheOrd}{\letterche}
\newcommand{\Che}{\CheOrd_{p}}

\newcommand{\Z}{{\mathbb{Z}}}

\newcommand{\ord}{{\mathrm{ord }}}

\begin{document}
%\today	

\title{Zsigmondy's Theorem for Chebyshev polynomials}

\author{Stefan Bara\'{n}czuk} 
\address{Faculty of Mathematics and Computer Science, Adam Mickiewicz University\\ul. Uniwersytetu Pozna\'{n}skiego 4, 61-416 Pozna\'{n}, Poland}
\email{stefbar@amu.edu.pl}

\begin{abstract} For an integer $a$  consider the sequence $(T_{n}(a)-1)_{n=1}^{\infty}$ defined by the Chebyshev polynomials $T_{n}$. We list all pairs  $(n,a)$ for which the  term $T_{n}(a)-1$ has no primitive prime divisor. 
\end{abstract}



\maketitle

%\linenumbers

There is an intriguing link between the sequence of the power maps $x^{n}$ and the sequence of the Chebyshev polynomials $T_{n}(x)$, defined either by the property $T_{n}(\cos(\theta))=\cos(n\theta)$, or recursively $T_{0}(x)=1$, $T_{1}(x)=x$, $T_{n+2}(x)=2xT_{n+1}(x)-T_{n}(x)$ (our reference for the Chebyshev polynomials is \cite{Riv}).
For example, they both satisfy the composition identity, which we state here for the Chebyshev polynomials:
\[T_{n}(T_{m}(x))=T_{m}(T_{n}(x))=T_{mn}(x).\]
Furthermore, the celebrated Julia-Ritt result says that if two polynomials commute under composition, then either both are iterates of the same polynomial, or both are in a sense similar to either Chebyshev polynomials or power maps.

There are also some number theoretic properties shared by both sequences (see Section 5.3. in \cite{Riv}). In this paper we investigate such property -- namely, we prove the Chebyshev polynomials analogue of Zsigmondy's Theorem. 

Zsigmondy's Theorem says for which natural numbers $a,n>1$ there is a prime divisor $p$ of $a^{n}-1$ that does not divide any of the numbers $a^{d}-1$,  $d<n$ (such primes are called \textit{primitive prime divisors}) or, equivalently, there is a prime number $p$ such that the multiplicative order $\ord_{p}(a)$ equals $n$. 

The above mentioned similarities between the power maps and the Chebyshev polynomials evoke the question whether we could replace $a^{n}$ in Zsigmondy's Theorem by $T_{n}(a)$. Our answer is  Theorem \ref{Satz 2}; in order to formulate it we have to introduce the following Chebyshevian analogue of the multiplicative order. For an integer $a$ and a prime number $p$ denote 
\[\Che (a)=\min \left\lbrace m\in\Z_{>0}\, \colon \, T_{m}(a) \equiv  1 \mod p \right\rbrace ;\]
this quantity always exists by Lemma \ref{LFT}.

\begin{thm}\label{Satz 2}
	Let $n$ and $a$  be integers such that $n>0$.  There exists a prime number $p$ such that $n=\Che(a)$,
	%$T_{n}(a) = 1 \mod p$ and $T_{k}(a) \ne 1 \mod p$ for every positive integer $k<n$, 
	except in the following cases:
	\begin{itemize}
		\item $a=1$ and $n>1$,
		\item $a=0$ and $n\notin\left\lbrace 2,4\right\rbrace $,
		\item $a=-1$ and $n>2$,
		\item $n=1$ and $a\in \left\lbrace0,2 \right\rbrace $,
		\item $n=2$ and $a=\pm 2^{\alpha-1}-1$,
		\item $n=3$ and $a=\tfrac{1}{2}(\pm 3^{\alpha}-1)$,
		\item $n=4$ and $a=\pm 2^{\alpha-1}$,
		\item $n=6$ and $a=\tfrac{1}{2}(\pm 3^{\alpha}+1)$,
	\end{itemize} 	
where $\alpha$ runs through positive integers.
\end{thm}
The proof, that occupies the reminder of the article, takes as a model the $\textit{einfacher Bewais}$ of Zsigmondy's Theorem presented in \cite{HL}.
% (compare our Theorem \ref{Satz 1} with Satz 1). 

\begin{prop}\label{RivExer} (Exercise 1.1.5 in \cite{Riv}) If $a,b$ are nonnegative integers, then
\[(T_{a+b}(x)-1)(T_{\left|a-b \right| }(x)-1)=(T_{a}(x)-T_{b}(x))^2.\]
\end{prop}

%\begin{prop}\label{T on x plus one} For every natural number $n$ and every real number $x$
% \[T_{n}(x+1)=1+n^2x+\ldots\]\end{prop}

The following lemma is the analogue of Fermat's little theorem for Chebyshev polynomials.
\begin{lem}\label{LFT} Let $p$ be an odd prime number. For every integer $x$
	\[ T_{p-1}(x)\equiv 1 \mod p \text{  \,\,\, or \,\,\,  }  T_{p+1}(x)\equiv 1 \mod p.\]
For every integer $x$	
\[ T_{2}(x)\equiv 1 \mod 2. \]
\end{lem}
	
\begin{proof}
	 If $p$ is an odd prime then we have (cf. (5.32) in \cite{Riv})
	 \[T_{p}(x)\equiv T_{1}(x) \mod p\]
and  
\[(T_{p+1}(x)-1)(T_{p-1}(x)-1)=(T_{p}(x)-T_{1}(x))^2\]	
by Proposition \ref{RivExer}.

For every integer $x$ 
\[T_{2}(x)=2x^{2}-1 \equiv 1 \mod 2.\]
\end{proof}

\begin{lem}\label{ChebOrd}
	Let $p$ be a prime number and $x$ be an integer. Let $m$ be the smallest positive integer such that $T_{m}(x) \equiv  1 \mod p$. 
	Then $T_{n}(x) \equiv  1 \mod p$ for a positive integer $n$ if and only if $m \mid n$. 
\end{lem}

\begin{proof} ($\Leftarrow$) For every $k$ we have $T_{k}(1)=1$. Thus if $n=km$ then  $T_{n}(x)=T_{k}(T_{m}(x)) \equiv T_{k}(1) \equiv  1 \mod p$ by  the composition identity.
	
($\Rightarrow$)	Let $r=n-km$ be the remainder obtained upon dividing $n$ by $m$. In particular, $r<m$. Suppose $r> 0$.
Putting $a=km$ and $b=r$ in Proposition \ref{RivExer} we get
\[(T_{n}(x)-1)(T_{\left|km-r \right| }(x)-1)=\left( (T_{km}(x)-1)-(T_{r}(x)-1)\right) ^2.\]
Arguing as in the ($\Leftarrow$) part of the proof we get $T_{km}(x) - 1\equiv 0 \mod p$. Since  $T_{n}(x) - 1 \equiv 0 \mod p$, we have $T_{r}(x) - 1\equiv 0 \mod p$ by the above identity. This contradicts the minimality of $m$.
\end{proof}

Lemmas \ref{LFT} and \ref{ChebOrd} together with $T_{1}(x)=x$  give the following.

\begin{lem}\label{prop Che}
	If $x$ is an integer and $p$ is an odd prime number then $\Che (x)$ divides $p-1$ or $p+1$. In particular, $\Che (x)$ and $p$ are coprime.
	If $x$ is odd then $\CheOrd_{2}(x)=1$. If $x$ is even then $\CheOrd_{2}(x)=2$.
\end{lem}

The key tool of the proof of Zsigmondy's Theorem is the factorization of polynomials $x^{n}-1$ into cyclotomic polynomials (our reference for them is \cite{HLB}). The following lemma describes its analogue for Chebyshev polynomials.

\begin{lem}\label{inrtocyc}
	For every $n\ge 1$
	\[T_{n}(x)-1=\prod_{d \mid n}\C_{d}^{\sigma_{d}}(x)\]
	where $\C_{1}(x)=x-1$ and for $d\ge 2$
	\[\C_{d}(x)=\prod_{
	\substack{
	1\le k \le \frac{d}{2}\\
	\gcd(k,d)=1}}2(x-\cos\tfrac{2 k \pi}{d})\]
and \[\sigma_{d}= \left\{
\begin{array}{rl}
1 & \text{if } d=1,2,\\
2 & \text{if } d >2.
\end{array} \right.\]
\end{lem}
\begin{proof}
	For every $n\ge 1$ the local maxima of $T_{n}(x)$ are exactly at $\cos\tfrac{2 k \pi}{n}$, $1\le k < \frac{n}{2}$, and they all are roots of multiplicity $2$ of $T_{n}(x)-1$. Besides those points, $T_{n}(x)=1$ only for $x=1$ and arbitrary $n$, and for $x=-1$ and even $n$ (see Section 1.2. in \cite{Riv}). 
\end{proof}

The significance of $\C_{n}(x)$ can be seen at a glance: it is exactly the factor that distinguishes $T_{n}(x)-1$ from all $T_{d}(x)-1$, $d\mid n$, $d<n$; see Corollary \ref{if T then Omega}.
%We list the needed properties of  $\C_{n}(x)$ in the following proposition. 

\begin{prop}\label{omegas prop b} Let $m, n$ be positive integers. Then the polynomial $\C_{mn}$ is a divisor of $\C_{n}(T_{m}(x))$.
If moreover $n\ge 3$ and every prime divisor of $m$ divides $n$, then $\C_{mn}=\C_{n}(T_{m}(x))$.
\end{prop}
\begin{proof}
	Let $\alpha$ be a zero of $\C_{mn}$. We have $\alpha=\cos\tfrac{2 k \pi}{mn}$ for some $k$ coprime to $mn$, with $1\le k \le \frac{mn}{2}$. Since $T_{m}(\cos(\theta))=\cos(m\theta)$, we get $T_{m}(\alpha)=\cos\tfrac{2 k \pi}{n}=\cos\tfrac{2 \bar{k} \pi}{n}=\cos\tfrac{2 (n-\bar{k}) \pi}{n}$, where $\bar{k}$ denotes $k \bmod n$. Since both the numbers $\bar{k}$ and $n-\bar{k}$ are coprime to $n$, and one of them is less than or equal to  $\tfrac{n}{2}$, we get that $T_{m}(\alpha)$ is a zero of $\C_{n}$. So all zeros of $\C_{mn}$ are zeros of $\C_{n}(T_{m}(x))$. Since $\C_{mn}$ has only simple zeros, we get that $\C_{mn}$ is a divisor of $\C_{n}(T_{m}(x))$.
	
	Now suppose that $n\ge 3$ and every prime divisor of $m$ divides also $n$. 
	If $d\ge 3$ then the degree of $\C_{d}$ is ${\varphi(d)}/{2}$ and its leading coefficient is $2^{{\varphi(d)}/{2}}$.  If every prime divisor of $m$ divides also $n$, then $\varphi(mn)=m\varphi(n)$. Thus elementary computations show that $\C_{mn}$ and  $\C_{n}(T_{m}(x))$ have the same degree and the same leading coefficient.
	\end{proof}

\begin{prop}\label{omega odd} Let $n$ be an odd natural number. Then $\C_{n}(0)=\pm 1$. If moreover $n\ge 3$ then $\C_{2n}(x)=\pm\C_{n}(-x)$. 
	\end{prop}
\begin{proof}
	The proof of the first statement is by induction on $n$. We have $\C_{1}(x)=x-1$, so $\C_{1}(0)=-1$. Suppose that for every odd natural number $d$ such that $1\le d<n$ we have $\C_{d}(0)=\pm 1$. 
	Since $T_{n}(0)=0$ for odd $n$, we have
	\[-1=T_{n}(0)-1=\prod_{d \mid n}\C_{d}^{\sigma_{d}}(0)=\C_{1}(0)\cdot\prod_{d \mid n,\, 1<d<n}\C_{d}^{2}(0)\cdot\C_{n}^{2}(0)= - \C_{n}^{2}(0).\]
	Thus we get $\C_{n}^{2}(0)=1$.
	
	Now suppose $n\ge 3$. We have $\deg \C_{2n}={\varphi(2n)}/{2}={\varphi(n)}/{2}=\deg \C_{n}$. The same  shows that $\C_{2n}$ and $\C_{n}$ have the same leading coefficient, namely $2^{{\varphi(n)}/{2}}$. It remains to examine the zeros. We have
	$-\cos\tfrac{2 k \pi}{n}=\cos\tfrac{2 (n-2k) \pi}{2n}$. Denote $l=n-2k$. The conditions $1\le k \le \frac{n}{2}, \,
	\gcd(k,n)=1$ are equivalent to $1\le l \le \frac{2n}{2}, \,
	\gcd(l,2n)=1$. Hence $\C_{2n}(x)$ and $\C_{n}(-x)$ have the same set of zeros.
\end{proof}

\begin{prop}\label{poly square int}
	Let $\mathbb{K}$ be a field of characteristic $0$. Suppose that $P\in \mathbb{K}[x]$, $P(0)=\pm 1$, and $P^{2}\in \Z[x]$. Then $P\in \Z[x]$.
\end{prop}
\begin{proof}
	Assume without loss of generality that $P(0)= 1$.  Put $P(x)=\sum_{i=0}^{\infty}c_{i}x^{i}$. Since for every $k\ge 1$ the coefficient of $x^{k}$ in $P^{2}(x)$ equals $2c_{k}+\sum_{0<i<k}c_{i}c_{k-i}$, we get by induction on $k$ that each $c_{k}$ is a rational number with denominator equal to $1$ or to a power of $2$. Suppose $P\notin \Z[x]$. Thus $P(x)=1+\tfrac{Q(x)}{2^{e}}$, where $Q(x)=\sum_{i=1}^{\infty}d_{i}x^{i}\in \Z[x]$ with some $d_{j}$ being odd, and $e$ is a positive integer. The coefficient of $x^{2j}$ in $Q^{2}(x)$ equals $d_{j}^{2} + 2\sum_{0<i<j}d_{j+i}d_{j-i}$, so it is an odd integer. Thus the coefficient of $x^{2j}$ in $2^{e+1}Q(x)+Q^{2}(x)$ is also odd. But we have $P^{2}(x)=1+\tfrac{2^{e+1}Q(x)+Q^{2}(x)}{2^{2e}}$, so $e=0$, a contradiction.  
\end{proof}

\begin{lem}\label{Omega integral} $\C_{n}\in \Z[x]$. 
\end{lem}
\begin{proof}
	First we prove the lemma for odd $n$. We use induction. $\C_{1}\in \Z[x]$. Let $n>1$. Suppose that for every odd natural number $d$ such that $1\le d<n$ we have $\C_{d}\in \Z[x]$. We have $T_{n}(x)-1=\prod_{d \mid n}\C_{d}^{\sigma_{d}}(x)=\C_{n}^{2}(x)g(x)$, where $g\in \Z[x]$. Put $\C_{n}^{2}(x)=\sum a_{i}x^{i}$ and $g(x)=\sum b_{i}x^{i}$. We have $a_{0}=1$ by Proposition \ref{omega odd}, and furthermore we get $b_{0}=-1$  since $T_{n}(0)=0$ for odd $n$. Fix  $i\le n$ and assume that $a_{j}\in \Z$ for every $j<i$. Since $a_{i}b_{0}+a_{i-1}b_{1}+\ldots +a_{0}b_{i}\in \Z$ as the coefficient of $x^{i}$ in $T_{n}(x)-1$, we have $a_{i}\in \Z$. Thus $\C_{n}^{2}\in \Z[x]$, and by Propositions \ref{omega odd} and \ref{poly square int} we get $\C_{n}\in \Z[x]$. Hence the lemma is proved for odd n.    
	
	If $n$ is the product of $2$ and an odd natural number greater than or equal to $3$, we use Proposition \ref{omega odd}.
	
	We directly compute $\C_{2}(x)=2(x+1)$ and $\C_{4}(x)=2x$.
	
	For $n\ge 8$ having the form $n=2^{\alpha}l$ where $\alpha \ge 2$ and $l$ is odd, we use Proposition \ref{omegas prop b} with 
	$\C_{n}=\C_{2l}(T_{2^{\alpha-1}}(x))$.	
\end{proof}

\begin{cor}\label{if T then Omega} Let $x$ be an integer.
	Every primitive prime divisor of $T_{n}(x)-1$ divides $\C_{n}(x)$.
\end{cor}

\begin{proof}
	Immediate from Lemmas \ref{inrtocyc} and  \ref{Omega integral}.
\end{proof}

\begin{prop}\label{T on x plus one minus one by x}
	For every natural number $n$ and every nonzero real number $x$ 
	\[\frac{T_{n}(x+1)-1}{x}=n^2+\frac{n^2(n^2-1)}{6}x+\frac{n^2(n^2-1)(n^2-4)}{90}x^{2}+\ldots \in \Z[x],\]
	where the dots denote irrelevant terms. 
\end{prop}

\begin{proof}
	The formula 
	\[T_{n}(x+1)=1+n^2 x+\frac{n^2(n^2-1)}{6}x^{2}+\frac{n^2(n^2-1)(n^2-4)}{90}x^{3}+\ldots\]
	can be proved by induction on $n$. 
\end{proof}

\begin{rem} One can observe that
	\[T_{n}(x+1)=1+\sum_{k=1}^{n}\frac{2^{k}\prod_{i=0}^{k-1}(n^{2}-i^{2})}{(2k)!}x^{k}.\]
\end{rem}	

\begin{thm}\label{Satz 1}
Let $n$ and $a$  be integers such that $n>1$ and $\left| a\right| >1$. If $p$ is a prime number dividing $\Cn$ then there exists a nonnegative integer $i$ such that $n=\Che(a)p^{i}$. If $i>0$, then $p$ is the greatest prime divisor of $n$. If moreover $p^2 \mid \Cn$ then either $p=2$ and $n\in\left\lbrace 2,4\right\rbrace$, or $p=3$ and $n\in\left\lbrace 3,6\right\rbrace $.   
\end{thm}

\begin{proof} $\Cn$ is a divisor of $T_{n}(a)-1$, so  $T_{n}(a)-1 \equiv 0 \mod p$. Hence $\Che(a)\mid n$ by Lemma \ref{ChebOrd}, and we can write $n=\Che(a)p^{i}w$, where $w$ is a natural number not divisible by $p$. Denote $r=\Che(a)p^{i}$. Since $\Che(a)\mid r$, we get by Lemma \ref{ChebOrd} that $T_{r}(a)-1 \equiv  0\mod p$. 
Compute
\[\frac{T_{n}(a)-1}{T_{r}(a)-1}= \frac{T_{w}((T_{r}(a)-1)+1)-1}{T_{r}(a)-1}\equiv w^2 \mod p,\]
where the congruence is obtained by putting $n=w$ and $x=T_{r}(a)-1$ in Proposition \ref{T on x plus one minus one by x}.
Suppose $w>1$. This implies $r<n$. Hence $\Cn$ is a divisor of $\tfrac{T_{n}(a)-1}{T_{r}(a)-1}$ by Lemma \ref{inrtocyc}. But $p \mid \Cn$, so we get $p\mid w^2$, contrary to the definition of $w$. Thus $w=1$ and $n=\Che(a)p^{i}$. \\

Suppose $i>0$. \\

Lemma \ref{prop Che} asserts that $\Che(a)$ divides one of the numbers $p-1 ,p , p+1$, and that $\CheOrd_{2}(a)\le 2$. Thus $p$ is the greatest prime divisor of $n$.\\

Define $s=\Che(a)p^{i-1}$. By Lemma \ref{ChebOrd} we have $T_{s}(a)-1 \equiv  0\mod p$. Assume $p\ge 5$. Compute 
\[ \frac{T_{n}(a)-1}{T_{s}(a)-1}= \frac{T_{p}((T_{s}(a)-1)+1)-1}{T_{s}(a)-1}\equiv p^2 \mod p^3, \]
where the congruence is obtained by putting $n=p$ and $x=T_{s}(a)-1$ in Proposition \ref{T on x plus one minus one by x}. 
Since $s\mid n$ and $s<n$, we get by Lemma \ref{inrtocyc} that $\Cn^{\sigma_{n}}$ is a divisor of $\tfrac{T_{n}(a)-1}{T_{s}(a)-1}$. Assume $p^2 \mid \Cn$. If $n> 2$ then $\sigma_{n}=2$ and we get  a contradiction with the above computation. Thus  we have $p\in \left\lbrace 2,3\right\rbrace $ or $n=2$. The latter returns us to the case $p=2$.

Consider first the case when $p=2$. Since $p$ is the greatest prime divisor of $n$, we have $n=2^{i}$. So $4 \mid \C_{2^{i}}(a)$. 
Since $\C_{4}(x)=2x$, we get by Proposition \ref{omegas prop b} that  $\C_{2^{i}}(a)=2 T_{2^{i-2}}(a)$ for $i>1$. 
Thus for $i>1$ we get $2 \mid T_{2^{i-2}}(a)$. But $T_{2}(x)=2x^{2}-1$, so using the composition identity we have $T_{2^{i-2}}(a)\equiv 1 \mod 2$ for $i\ge 3$. So $i\in \left\lbrace 1,2 \right\rbrace $. Thus $n \in \left\lbrace 2,4\right\rbrace $.

%But  $2T_{2^{i-2}}(a)\equiv 2 \mod 4$ for $i>2$ (for $i=3$ we have $2T_{2}(x)=4x^{2}-2$, and for higher $i$ use the composition identity). Thus $i=1$ or $i=2$. \\
%If $i=1$ then $n=2$. If $i=2$ then $4 \mid \C_{4}(a)=a$, so we have $n=4$ and $a = 0 \mod 4$.  

Now let $p=3$. Since $p$ is the greatest prime divisor of $n$, we have $n=2^{j}3^{i}$ with $i\ge1$. 

Consider first the case when $j=0$. The only zero in $\Z/9\Z$ of the polynomial $\C_{3}(x)=2x+1$ is $x=4$. The image of  $T_{3}$ on   $\Z/9\Z$ equals  $\left\lbrace 0, 1, 8\right\rbrace $. So by Proposition \ref{omegas prop b} and the composition identity we get that $9 \mid \Cn$ implies $n=3$.

Now consider the case when $j\ge1$. The only zero in $\Z/9\Z$ of the polynomial $\C_{6}(x)=2x-1$ is $x=5$. The image of $T_{2}$  on   $\Z/9\Z$ equals $\left\lbrace1, 4, 7, 8 \right\rbrace $, and the image of $T_{3}$ equals $\left\lbrace 0, 1, 8\right\rbrace $. So by Proposition \ref{omegas prop b} and the composition identity we get that $9 \mid \Cn$ implies $n=6$.

Thus  $n\in \left\lbrace 3,6\right\rbrace $. 
%Now let $n=2$. We get that $\Che(a)=1$. Thus $p^2 \mid \C_{2}(a)=2(a+1)$ and $p \mid T_{1}(a)-1=a-1$. So $p=2$.
\end{proof}

\begin{cor}\label{wniosek}
Let $n$ and $a$  be integers such that $n>1$ and $\left| a\right| >1$. The following statements are equivalent.
\begin{enumerate}[label={\text{S}\arabic*}]
	\item \label{TFAE1} A prime number $p$ such that $n=\Che(a)$ does not exist.
	\item \label{TFAE2} $\left| \Cn \right|$  is a power of a prime number dividing $n$.
	\item \label{TFAE3} $\left| \Cn \right|$  is a power of the greatest prime divisor of $n$.
\end{enumerate}  
%Suppose that the set $S$ of prime divisors of $\Cn$ has more than one element. Then  $n=\Che(a)$ for some $p\in S$. 
%The number $d=\gcd(\Cn,n)$ equals either $1$, or $4$, or the greatest prime divisor of $n$. A prime number $p$ such that $T_{n}(a) = 1 \mod p$ and $T_{k}(a) \ne 1 \mod p$ for every positive integer $k<n$ does not exists if and only if $d>0$ and $\Cn$ is a power of $d$.
\end{cor}

\begin{proof}
	By the definition of $\C$'s we have $\left| \Cn \right| >1$. \\
	
	First we show that \ref{TFAE1} implies \ref{TFAE2}. Let  $p_{1}, p_{2}$ be prime divisors of  $\Cn$. By Theorem \ref{Satz 1} we have $n=\CheOrd_{p_{1}}(a) p_{1}^{i_{1}}=\CheOrd_{p_{2}}(a) p_{2}^{i_{2}}$. Since neither $\CheOrd_{p_{1}}(a)$ nor $\CheOrd_{p_{2}}(a)$ equals $n$, we get $i_{1}, i_{2}>0$. Thus both $p_{1}, p_{2}$ are the greatest prime divisor of $n$, i.e., $p_{1}=p_{2}$. So \ref{TFAE2} holds.\\
	
	Now we show that \ref{TFAE2} implies \ref{TFAE1}.
	
	Let $\left| \Cn \right|$ be a power of a prime number $p$ coprime to $n$. By Theorem \ref{Satz 1} we have $\Che(a)=n$.
	
	%If $\left| \Cn \right| =1$ then a primitive prime divisor does not exist by Corollary \ref{if T then Omega}.\\
	
	Now let $\left| \Cn \right|$  be a power of a prime number $p$ dividing $n$. By Corollary \ref{if T then Omega}, if a primitive prime divisor exists, it equals $p$. Consider two cases.
	
	First, let $p$ be odd. By Lemma \ref{prop Che} we have that $\Che(a)$ is coprime to $p$. So by Theorem \ref{Satz 1} we get that $n=\Che(a) p^{i}$ with $i>0$. Thus $n\ne \Che(a)$.  
	
	Now let $p=2$. By Lemma \ref{prop Che} the possible values of  $\CheOrd_{2}(a)$ are $1$ or $2$. Hence if $n=\CheOrd_{2}(a)$ then $n=2$. But $\C_{2}(a)=2(a+1)$ can be a power of $2$ only for odd $a$, and for such $a$ we have $\CheOrd_{2}(a)=1$.
	
	Thus in both cases a primitive prime divisor does not exist.\\
	
	Hence the equivalence of the statements \ref{TFAE1} and \ref{TFAE2} is established, and Theorem \ref{Satz 1} immediately gives their equivalence to \ref{TFAE3}.  
	
\end{proof}

\begin{proof}[Proof of Theorem \ref{Satz 2}] Let $n>1$ and $\left| a\right| >1$. Suppose that a primitive prime divisor does not exist, and that $n\notin\left\lbrace2,3,4,6 \right\rbrace $. By Corollary \ref{wniosek} \ref{TFAE2} and Theorem \ref{Satz 1} we get $\left| \Cn\right| =p$ for a prime number $p$ dividing $n$. Hence  
%For  $1\le k < \frac{n}{2}$  we have $\left|  a-\cos\tfrac{2 k \pi}{n} \right| > \left| a\right| -1  $. Since $p\mid n$, we get $p-1=\varphi(p)\mid \varphi(n)$. 
	\[p=\left| \Cn\right| = \prod_{
		\substack{
			1\le k \le \frac{n}{2}\\
			\gcd(k,n)=1}}2\left| a-\cos\tfrac{2 k \pi}{n}\right| >\left( 2 (\left| a\right| -1) \right) ^{\frac{\varphi(n)}{2}} \ge \left(2 (\left| a\right| -1) \right) ^{\frac{p-1}{2}}, \]
so we have two bounds:
\[p>\left( 2 (\left| a\right| -1) \right) ^{\frac{\varphi(n)}{2}}\]	
and
\[p>\left(2 (\left| a\right| -1) \right) ^{\frac{p-1}{2}}.\]	
The second bound implies $\left| a\right| =2$ and $p\in \left\lbrace 2,3,5\right\rbrace$. By Corollary \ref{wniosek} \ref{TFAE3} we have  for $p=2$  that $n$ is divisible by $8$, and for $p=3$ that $n$ is divisible by $9$ or $12$.  Thus for $p \in \left\lbrace2,3 \right\rbrace $ we have $\varphi(n)\ge 4$  and we get a contradiction with the first bound. Hence $p=5$.
Suppose $5^{2}\mid n$. This implies that $\varphi(n)\ge 20$ and again we get a contradiction with the first bound. 
%This means that $p\mid \tfrac{n}{p}$. Thus by Proposition \ref{omegas prop b} we have $\C_{n}(x)=\C_{\frac{n}{p}p}(x)=\C_{\frac{n}{p}}(T_{p}(x))$. Using this formula and the fact that $\left| T_{p}(a)\right|>1$,  we get as above \[p=\left| \C_{n}(a)\right| =\left| \C_{\frac{n}{p}}(T_{p}(a))\right| >\left( 2 (\left| T_{p}(a)\right| -1) \right) ^{\frac{\varphi\left( \frac{n}{p}\right) }{2}}\ge \left( 2 (\left| T_{p}(a)\right| -1) \right) ^{\frac{p-1 }{2}}.\]	 But $\left| T_{3}(\pm 2)\right| -1=25$ and $\left| T_{5}(\pm 2)\right| -1=361$, a contradiction. 
Hence $n=5\cdot \CheOrd_{5}(a)$ by Theorem \ref{Satz 1}. 

Assume $a=2$. We have  $\CheOrd_{5}(2)=3$, so $n=15$. But $\C_{15}(2)=5\cdot 29$ is not a prime number, a contradiction.	

Assume $a=-2$. We have $\CheOrd_{5}(-2)=6$, so $n=30$. But $\C_{30}(-2)=5\cdot 29$ is not a prime number, a contradiction.	

%Now let $\Cn$ be a power of of $2$. Suppose $n\notin\left\lbrace2,3,4,6 \right\rbrace $. By Lemma \ref{prop Che} and Theorem \ref{Satz 1} we get that $n=2^{i}$, $i \ge 3$, and $\Cn=2$. We use again the identity $\C_{2^{i}}=2 T_{2^{i-2}}$. For $a\ge 2$ the sequence $T_{n}(a)$ is strictly increasing and $T_{0}(a)=1$. Thus $\Cn>2$, a contradiction.\\

Hence the exceptional cases can appear only for $n\in\left\lbrace2,3,4,6 \right\rbrace $. We obtain the values of corresponding $a$'s by examining $\C_{2}(a)=2(a+1)$, $\C_{3}(a)=2a+1$, $\C_{4}(a)=2a$, and $\C_{6}(a)=2a-1$, according to Corollary \ref{wniosek} \ref{TFAE3}.  \\

It remains to analyse the trivial cases.

$T_{1}(a)-1=\pm1$ if and only if $a \in \left\lbrace 0,2\right\rbrace $.

Evaluating the sequence $(T_{n}(a)-1)_{n=1}^{\infty}$ at $a=1$ we get $0, 0, \ldots$, so every prime number is a primitive prime divisor of the first term.

Evaluating the sequence $(T_{n}(a)-1)_{n=1}^{\infty}$ at $a=0$ we get $-1, -2, -1, 0, -1, -2, -1, 0, \ldots$, so the first and the third term have no primitive prime divisor, $2$ is the primitive prime divisor of the second term, and every odd prime number is a primitive prime divisor of the fourth term.

Evaluating the sequence $(T_{n}(a)-1)_{n=1}^{\infty}$ at $a= -1$ we get $-2, 0, -2, 0, \ldots$, so $2$ is the primitive prime divisor of the first term, and every odd prime number is a primitive prime divisor of the second term.
\end{proof}

\section*{Acknowlegements} We are grateful to Bart{\l}omiej Bzd\k{e}ga for suggestions regarding the proof of Proposition \ref{poly square int}, and to Bartosz Naskr\k{e}cki for insightful remarks on the first version of the manuscript.

The Magma online calculator has been more than helpful. We have found the formulas for the coefficients in Proposition \ref{T on x plus one minus one by x} thanks to OEIS database.

%-----------------BIBLIOGRAPHY----------

\bibliographystyle{plain}

\end{document}